\documentclass{article}

\usepackage{amsmath,amssymb,amsthm}
\usepackage{graphicx}
\usepackage{caption}
\usepackage{subcaption}
\usepackage{url}
\usepackage{a4wide}

\newtheorem{theorem}{Theorem}

\theoremstyle{definition}
\newtheorem{definition}[theorem]{Definition}

\begin{document}

\title{Seasonality effects on Dengue: basic reproduction number,\\
sensitivity analysis and optimal control\thanks{This is a preprint
of a paper whose final and definite form is published in
\emph{Mathematical Methods in the Applied Sciences}, ISSN 0170-4214
(see {\tt http://dx.doi.org/10.1002/mma.3319}).
Paper submitted 22/July/2014; revised 11/Sept/2014;
accepted for publication 12/Sept/2014.}}

\author{Helena Sofia Rodrigues$^{1, 2}$\\
{\tt \small sofiarodrigues@esce.ipvc.pt}
\and M. Teresa T. Monteiro$^3$\\
{\tt \small tm@dps.uminho.pt}
\and Delfim F. M. Torres$^1$\\
{\tt \small delfim@ua.pt}}


\date{$^1$\mbox{Center for Research and Development in Mathematics and Applications (CIDMA)},
Department of Mathematics, University of Aveiro, 3810--193 Aveiro, Portugal\\[0.3cm]
$^2$School of Business Studies, Viana do Castelo Polytechnic Institute,\\
Avenida Miguel Dantas, 4930--678 Valen\c{c}a, Portugal\\[0.3cm]
$^3$Algoritmi R\&D Center, Department of Production and Systems,\\
University of Minho, Campus de Gualtar, 4710--057 Braga, Portugal}

\maketitle


\begin{abstract}
Dengue is a vector-borne disease transmitted from an infected human to
an \emph{Aedes} mosquito, during a blood meal. Dengue is still a major
public health problem. A model for the disease transmission is presented,
composed by human and mosquitoes compartments. The aim is to simulate
the effects of seasonality, on the vectorial capacity and, consequently,
on the disease development. Using entomological information about the
mosquito behavior under different temperatures and rainfall, simulations
are carried out and the repercussions analyzed. The basic reproduction number
of the model is given, as well as a sensitivity analysis of model's parameters.
Finally, an optimal control problem is proposed and solved, illustrating
the difficulty of making a trade-off between reduction of infected individuals
and costs with insecticide.

\smallskip


\smallskip

\end{abstract}


\section{Introduction}
\label{sec:1}

Dengue is currently one of the most important viral diseases transmitted by mosquitoes
to humans in a world context. It is transmitted by  \emph{Aedes aegypti} and
\emph{Aedes albopictus} and is usually found in tropical and sub-tropical regions,
but some recent episodes also happened in Europe \cite{Gjenero-Margan2011,Ruche2010,carla2012}.
There are four different serotypes that can cause dengue fever. A human infected by
one serotype, when recovered, has total immunity for that one, and only has partial
and transient immunity for the other three serotypes.

The life cycle of the mosquito has four distinct stages: egg, larva, pupa and adult.
The first three stages take place in water, while air is the medium
for the adult stage. In urban areas, \emph{Aedes aegypti} breeds on water
collections. With increasing urbanization and crowded cities,
environmental conditions foster the spread of the disease that,
even in the absence of fatal forms, breed significant economic and
social costs (absenteeism, immobilization,
debilitation and medication) \cite{Derouich2006}.
Until a vaccine or drug for dengue is available, vector control operations that
eliminate adult mosquitoes and their larvae through breeding-source reduction
remain the only effective method \cite{Who2009}. However, vector control can
be expensive and time consuming, producing a huge economic burden on nations.

Dengue epidemiology is influenced by a complex
set of factors that include rapid urbanization and increase in population
density, capacity of healthcare systems, herd immunity and social behavior of the population.
However, temperature and rainfall are a key environmental determinant in shaping the
landscape of disease. They are critical to mosquito survival, reproduction
and development, and can influence mosquito presence and abundance \cite{CDC}.
Additionally, higher temperatures reduce the time required for the virus
to replicate and disseminate in the mosquito \cite{James2013,LiuHelmersson2014,Vezzani2004}.
Thus, it is important to create distinct simulations to predict
the effects of seasonality on the disease transmission.

The text is organized as follows. In Section~\ref{sec:2}, a mathematical model of the interaction
between humans and mosquitoes is formulated, and the basic reproduction number, $\mathcal{R}_{0}$,
is calculated. A sensitivity analysis of the parameters used is carried out taking into
account $\mathcal{R}_{0}$. Different simulations of the model are shown,
varying the temperatures of the region. In Section~\ref{sec:3}, the mathematical
model is restructured, using a periodic function for the birthrate of the mosquito,
fitting a region with dry and rainy seasons all over a year.
An optimal control problem is proposed in Section~\ref{sec:4},
using the information given in Section~\ref{sec:2}, in order to analyze
different bioeconomic approaches for the dengue disease.
The main conclusions are given in Section~\ref{sec:5}.


\section{The mathematical model}
\label{sec:2}

Taking into account the model presented in
\cite{Dumont2010,Dumont2008} and the considerations
of \cite{Sofia2009,Sofia2010c,Sofia2014}, a mathematical model
is here proposed.
It includes three epidemiological states for humans:
\begin{quote}
\begin{tabular}{lll}
$S_h(t)$ & --- & susceptible (individuals who can contract the disease);\\
$I_h(t)$ & --- & infected (individuals who can transmit the disease); and\\
$R_h(t)$ & --- & resistant (individuals who have been infected and have recovered).
\end{tabular}
\end{quote}
These compartments are mutually exclusive. There are two other state variables, related
to the female mosquitoes (male mosquitos are not considered because they do not bite
humans and consequently do not influence the dynamics of the disease):
\begin{quote}
\begin{tabular}{lll}
$S_m(t)$ & --- & susceptible (mosquitoes that can contract the disease); and\\
$I_m(t)$ & --- & infected (mosquitoes that can transmit the disease).
\end{tabular}
\end{quote}
In order to make a trade-off between simplicity and reality of the epidemiological model,
some assumptions are considered:
\begin{itemize}
\item there is no vertical transmission, that is,
an infected mosquito cannot transmit the disease to their eggs;

\item total human population $N_h$
is constant: $S_h(t)+I_h(t)+R_h(t) = N_h$ at any time $t$;

\item the mosquito population is also constant and proportional to human population,
that is, $S_m(t)+I_m(t)=N_m$, with $N_m= \kappa N_h$ for some constant $\kappa$;

\item the population is homogeneous, which means that every individual
of a compartment is homogeneously mixed with the other individuals;

\item immigration and emigration are not considered during the period under study;

\item homogeneity between host and vector populations, that is,
each vector has an equal probability to bite any host;

\item humans and mosquitoes are assumed to be born susceptible.
\end{itemize}
The system of differential equations is composed by human compartments
\begin{equation}
\label{cap6_ode1}
\begin{cases}
\frac{dS_h(t)}{dt}
= \mu_h N_h - \left(B\beta_{mh}\frac{I_m(t)}{N_h}+\mu_h\right)S_h(t)\\
\displaystyle\frac{dI_h(t)}{dt} = B\beta_{mh}\frac{I_m(t)}{N_h}S_h(t) -(\eta_h+\mu_h) I_h(t)\\
\displaystyle\frac{dR_h(t)}{dt} = \eta_h I_h(t) - \mu_h R_h(t)
\end{cases}
\end{equation}
coupled with mosquito compartments
\begin{equation}
\label{cap6_ode2}
\begin{cases}
\frac{dS_m(t)}{dt} = \mu_m N_m
-\left(B \beta_{hm}\frac{I_h(t)}{N_h}+\mu_m \right) S_m(t)\\
\displaystyle\frac{dI_m(t)}{dt} = B \beta_{hm}\frac{I_h(t)}{N_h}S_m(t)
-\mu_m I_m(t)
\end{cases}
\end{equation}
and subject to initial conditions
\begin{equation}
\label{initial_conditions}
\begin{gathered}
S_h(0)=S_{h0}, \quad  I_h(0)=I_{h0}, \quad R_h(0)=R_{h0},\\
S_{m}(0)=S_{m0}, \quad I_m(0)=I_{m0}.
\end{gathered}
\end{equation}


\subsection{Scenarios with temperature variation}

The dengue epidemic model makes use of the parameters described in Table~\ref{table_1}.
In this study, three simulations were considered, related to distinct vectorial capacity.
Temperature affects the behavior of vector: its population, biting rate, biting capacity,
incubation time, daily survival probability or mortality rate, and eggs hatching rate \cite{Watts1987}.
It is generally assumed that higher mean temperatures facilitate dengue transmission
because of faster virus propagation and dissemination within the vector. Vector competence,
the probability of a mosquito becoming infected and subsequently transmitting virus
after ingestion of an infectious blood meal, is generally positively associated with
temperature \cite{Carrington2013}. We only assume differences on transmission capacities
and mosquito lifespan.

The different values presented for Scenarios 1 and 2 are based
on \cite{LiuHelmersson2014}. The first scenario is concerned with a region where
the mean temperature is $14^{\circ}$C. The second one is related to a region where the mean
temperature is $26^{\circ}$C.
The third scenario is created to simulate mild climate. The authors had previously analyzed the outbreak
that occurred in Madeira island in October 2012, which has a mean temperature between $18^{\circ}$C and $24^{\circ}$C,
all over the year. The values used in this last scenario are based on \cite{Sofia2014}.
\begin{table}[ptbh]
\begin{center}
\caption{Parameters in the epidemiological model \eqref{cap6_ode1}--\eqref{cap6_ode2}.}
\label{table_1}
\scriptsize{
\begin{tabular}{lllllll}
\hline
Para-\- & Description & Range of values  & Value & Value  & Value & Source\\
meter& & in literature & Scenario 1 & Scenario 2 & Scenario 3 &\\
\hline
$N_h $ & Total population & & 112000 & 112000 & 112000 &\cite{censos2011}\\
$N_m $ & Total mosquito population & & $3\times N_h$ &  $3\times N_h$ & $3\times N_h$ &\cite{censos2011}\\
$B$ & Average daily biting (per day)& & 1/3 & 1/3 & 1/3 &\cite{Focks2000}\\
$\beta_{mh}$ & Transmission probability &  &  & & \\
 &  from $I_m$ (per bite)& [0.1, 1] & 0.12 & 0.99 & 0.2 &\cite{Focks2000,LiuHelmersson2014} \\
$\beta_{hm}$ & Transmission probability &  &  & & \\
 & from $I_h$ (per bite)& [0.1, 1] & 0.11 & 0.95 & 0.2 & \cite{Focks2000,LiuHelmersson2014} \\
$1/\mu_{h}$ & Average lifespan of humans & &  &  & \\
 &  (in days)& & $\frac{1}{79\times365}$ & $\frac{1}{79\times365}$ & $\frac{1}{79\times365}$ &\cite{censos2011} \\
$1/\eta_{h}$ & Average viremic period (in days)& [1/15, 1/4] & 1/7 & 1/7 & 1/7 &\cite{Chan2012}\\
$1/\mu_{m}$ & Average lifespan of adult & &  &  & \\
 &  mosquitoes (in days)& [1/45, 1/8]& 0.04 & 0.03 & 1/15& \cite{Focks1993,Harrington2001,LiuHelmersson2014,Freitas2007} \\
\hline
\end{tabular}}
\end{center}
\end{table}


\subsection{Stability and sensitivity analysis}

The model \eqref{cap6_ode1}--\eqref{cap6_ode2} has two nonnegative equilibria. Namely,
\begin{itemize}
\item a disease-free equilibrium $(S_{h1}, I_{h1}, R_{h1}, S_{m1}, I_{m1})=(N_h, 0, 0, N_m, 0)$;
\item an endemic equilibrium $(S_{h2}, I_{h2}, R_{h2}, S_{m2}, I_{m2})$ with
\begin{equation*}
\begin{split}
S_{h2}&=\frac{N_h^2 (B \beta_{hm} \mu_h+\mu_m (\eta_h+\mu_h))}
{B \beta_{hm} (B \beta_{mh} N_m+\mu_h N_h)},\\
I_{h2}&=-\frac{\mu_h N_h \left(\mu_m N_h (\eta_h+\mu_h)-B^2 \beta_{hm} \beta_{mh} N_m\right)}
{B \beta_{hm} (\eta_h+\mu_h) (B\beta_{mh} N_m+\mu_h N_h)},\\
R_{h2}&=-\frac{\eta_h N_h \left(\mu_m N_h (\eta_h+\mu_h)-B^2 \beta_{hm} \beta_{mh} N_m\right)}
{B \beta_{hm} (\eta_h+\mu_h) (B\beta_{mh} N_m+\mu_h N_h)},\\
S_{m2}&=\frac{\mu_m (\eta_h+\mu_h) (B \beta_{mh} N_m+\mu_h N_h)}
{B \beta_{mh} (B \beta_{hm} \mu_h+\mu_m (\eta_h+\mu_h))},\\
I_{m2}&=-\frac{\mu_h \left(\mu_m N_h (\eta_h+\mu_h)-B^2 \beta_{hm}\beta_{mh}N_m\right)}
{B \beta_{mh} (B \beta_{hm} \mu_h+\mu_m (\eta_h+\mu_h))}.
\end{split}
\end{equation*}
\end{itemize}
An important measure of transmissibility of the disease is given
by the basic reproduction number. It represents the
expected number of secondary cases produced in a completed
susceptible population, by a typical infected individual during
its entire period of infectiousness \cite{Hethcote2000}.

\begin{theorem}
\label{thm:1}
The basic reproduction number $\mathcal{R}_0$ associated to the
differential system \eqref{cap6_ode1}--\eqref{cap6_ode2} is given by
\begin{equation}
\label{eq:R0}
\mathcal{R}_0 = \left(\frac{B^2 \beta_{hm} \beta_{mh} N_m}{(\eta_h + \mu_h) \mu_m N_h}\right)^{\frac{1}{2}}.
\end{equation}
\end{theorem}

\begin{proof}
Similar to the one found in \cite{Sofia2013}.
\end{proof}

If $\mathcal{R}_0 < 1$, then, on average, an infected individual produces
less than one new infected individual over the course of its infectious period,
and the disease cannot grow. Conversely, if $\mathcal{R}_0 > 1$, then each individual
infects more than one person, and the disease invades the population.
Mathematically, $\mathcal{R}_0$ is a threshold for stability
of a disease-free equilibrium and is related to the peak and final size
of an epidemic \cite{Driessche2002}. If $\mathcal{R}_0 < 1$, then
the disease-free equilibrium is stable; otherwise,
if $\mathcal{R}_0 > 1$, then it is unstable.

In determining how best to reduce human mortality and morbidity due to dengue,
it is necessary to know the relative importance of the different factors
responsible for its transmission. The sensitivity indices of $\mathcal{R}_0 $,
related to the parameters in the model, are now calculated.

Sensitivity indices allow us to measure the relative change in a variable when a
parameter changes. The normalized forward sensitivity index of a variable with respect
to a parameter is the ratio of the relative change in the variable to the relative change in the parameter.
When the variable is a differentiable function of the parameter, the sensitivity index may
be defined as follows.

\begin{definition}[See \cite{Chitnis2008}]
The normalized forward sensitivity index of $\mathcal{R}_0$,
which depends differentiably on a parameter $p$, is defined by
\begin{equation}
\label{eq:sens:index}
\Upsilon_p^{\mathcal{R}_0}
=\frac{\partial \mathcal{R}_0}{\partial p}\times \frac{p}{\mathcal{R}_0}.
\end{equation}
\end{definition}
Given the explicit formula \eqref{eq:R0} for the basic reproduction number,
one can easily derive an analytical expression for the sensitivity of  $\mathcal{R}_0$ with
respect to each parameter that comprise it. The obtained values are in Table~\ref{sensitivity},
which presents the sensitivity indices for the baseline parameter values.
Note that the sensitivity index (column~2 of Table~\ref{sensitivity})
may be a complex expression, depending on the different parameters
of the system, but can also be a constant value, not depending on any parameter value.
Column~3 of Table~\ref{sensitivity} presents the values of the indices,
considering the parameter values of Table~\ref{table_1}.
\begin{table}
\begin{center}
\caption{Sensitivity indices \eqref{eq:sens:index} of \eqref{eq:R0}
evaluated at the baseline parameter values in Table~\ref{table_1}.}
\label{sensitivity}
\scriptsize{
\begin{tabular}{lll}
\hline
Parameter & Sensitivity index & Sensitivity index\\
& & for parameter values\\
\hline
$B$ & +1 &+1\\
$\beta_{mh}$ & +0.5 &+0.5 \\
$\beta_{hm}$ & +0.5 &+0.5\\
$\mu_{h}$ & $-\frac{\eta_h}{2 (\eta_h + \mu_h)}$ &-0.00012\\
$\eta_{h}$ &$-\frac{\mu_h}{2 (\eta_h + \mu_h)}$& -0.49988\\
$\mu_{m}$ & -0.5& -0.5 \\
\hline
\end{tabular}}
\end{center}
\end{table}
For example, $\Upsilon_{\beta_{mh}}^{\mathcal{R}_0} \equiv +0.5$ means that
increasing (or decreasing) $\beta_{mh}$ by $10\%$ increases (or decreases) always
$\mathcal{R}_0$ by $5\%$.
A highly sensitive parameter should be carefully estimated, because a small
variation in that parameter will lead to large quantitative changes.
An insensitive parameter, on the other hand, does not require as much effort to estimate,
because a small variation in that parameter will not produce large changes to the quantity
of interest. The results show that big changes in the parameters that affects
the basic reproduction number (except $\mu_h$) produce significant changes in $\mathcal{R}_{0}$,
and consequently, in the behavior of the disease development.
For the three scenarios we study, the basic reproduction number has the values
0.7698, 73.1322 and 0.6221, respectively. This means that if there is no change or
control for the disease, the outbreak will die out in a short period in Scenarios~1
and 3. In contrast, the disease will persist and will become endemic in the region
of Scenario~2.


\subsection{Numerical analysis}
\label{sec:2:1}

The software used in our simulations was \texttt{Matlab} with the routine \texttt{ode45}.
This solver is based on an explicit Runge--Kutta (4,5) formula, the Dormand--Prince pair.
That means the numerical solver \texttt{ode45} combines fourth and fifth order methods, both of
which are similar to the classical fourth order Runge--Kutta method. These vary the step
size, choosing it at each step in an attempt to achieve the desired accuracy.
We examine simulations of system \eqref{cap6_ode1}--\eqref{cap6_ode2},
considering final time $t_f=365$ days with the following initial values \eqref{initial_conditions}
for the differential equations:
\begin{equation*}
\begin{gathered}
S_h(0)=111991, \quad  I_h(0)=9, \quad R_h(0)=0,\\
S_{m}(0)=335000, \quad I_m(0)=1000.
\end{gathered}
\end{equation*}
Figure~\ref{ih_constant} shows the evolution of
infected human in the three scenarios, respectively.
Figure~\ref{ih_constant_high} presents more infected people,
because it corresponds to higher levels of disease
transmissibility due to higher temperatures.
Besides, it is also this scenario that reaches the peak of the disease faster,
while the third simulation has its higher transmission after 200 days.
A situation like this last one allows to have time to prepare the fight of the disease,
in terms of control measures and medical surveillance.

Besides the constraint of temperature effects,
the rainfall factor is also included in next section.
\begin{figure}
\centering
\begin{subfigure}[b]{0.4\textwidth}
\includegraphics[scale=0.45]{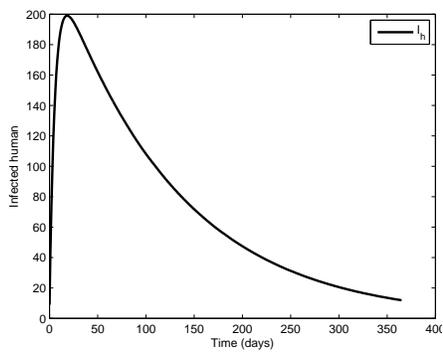}
\caption{Scenario 1}
\label{ih_constant_low}
\end{subfigure}
\begin{subfigure}[b]{0.4\textwidth}
\includegraphics[scale=0.45]{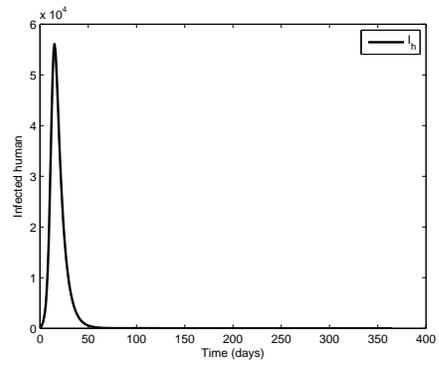}
\caption{Scenario 2}
\label{ih_constant_high}
\end{subfigure}
\begin{subfigure}[b]{0.4\textwidth}
\includegraphics[scale=0.45]{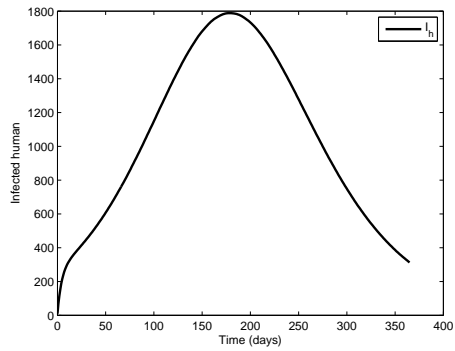}
\caption{Scenario 3}
\label{ih_constant_madeira}
\end{subfigure}
\caption{Infected human.}\label{ih_constant}
\end{figure}


\section{Mathematical model with seasonal variation of mosquito}
\label{sec:3}

In this section, we study the effect of rainfall on the pattern
of mosquito reproduction and hence the number of mosquitoes.
We maintain all the assumptions given before,
except assuming that birth and death rates are equal over time.
The seasonal effect in the modeling of virus transmission
is incorporated, allowing the total number of mosquitoes
to vary periodically with time.

Following \cite{James2013}, we include this seasonal
pattern in system \eqref{cap6_ode1}--\eqref{cap6_ode2}
by changing the birthrate of mosquitoes to a periodic function
\begin{equation*}
\mu_m\left( 1+\alpha \cos \left(\frac{2 \pi t}{365}\right)\right),
\end{equation*}
where $\mu_m$ is the per capita death rate of mosquitoes and $\alpha$ is the amplitude of the
seasonal variation, with $0 < \alpha < 1$. So, the differential equation in \eqref{cap6_ode2}
related to the susceptible mosquitos is transformed into
\begin{equation*}
\displaystyle\frac{dS_m(t)}{dt}
= \mu_m\left( 1+\alpha \cos \left(\frac{2 \pi t}{365}\right)\right)\left(S_m(t)+I_m(t)\right)
-\left(B \beta_{hm}\frac{I_h(t)}{N_h}+\mu_m \right) S_m(t).
\end{equation*}
This reformulated model has only the trivial equilibrium
$(\tilde{S}_{h1}, \tilde{I}_{h1}, \tilde{R}_{h1}, \tilde{S}_{m1}, \tilde{I}_{m1})=(N_h, 0, 0, 0, 0)$.

For numerical experiments, we considered $\alpha=0.3$. In this way,
the mosquito reproduction has its lowest value around 190 days after
the beginning of the year. Again, the routine
\texttt{ode45} of \texttt{Matlab} was used.
Figure~\ref{ih_mix} presents the simulations. The solid line shows
the situation described in Section~\ref{sec:2:1} with all the parameters fixed; in dashed line,
we represent the simulation with the periodic function. In all situations, the simulations with the
seasonal pattern present more infected people. In Scenario~1, the
peak of the disease with the periodic function is reached later,
while in Scenario~3, the situation is reversed.
\begin{figure}
\centering
\begin{subfigure}[b]{0.4\textwidth}
\includegraphics[scale=0.45]{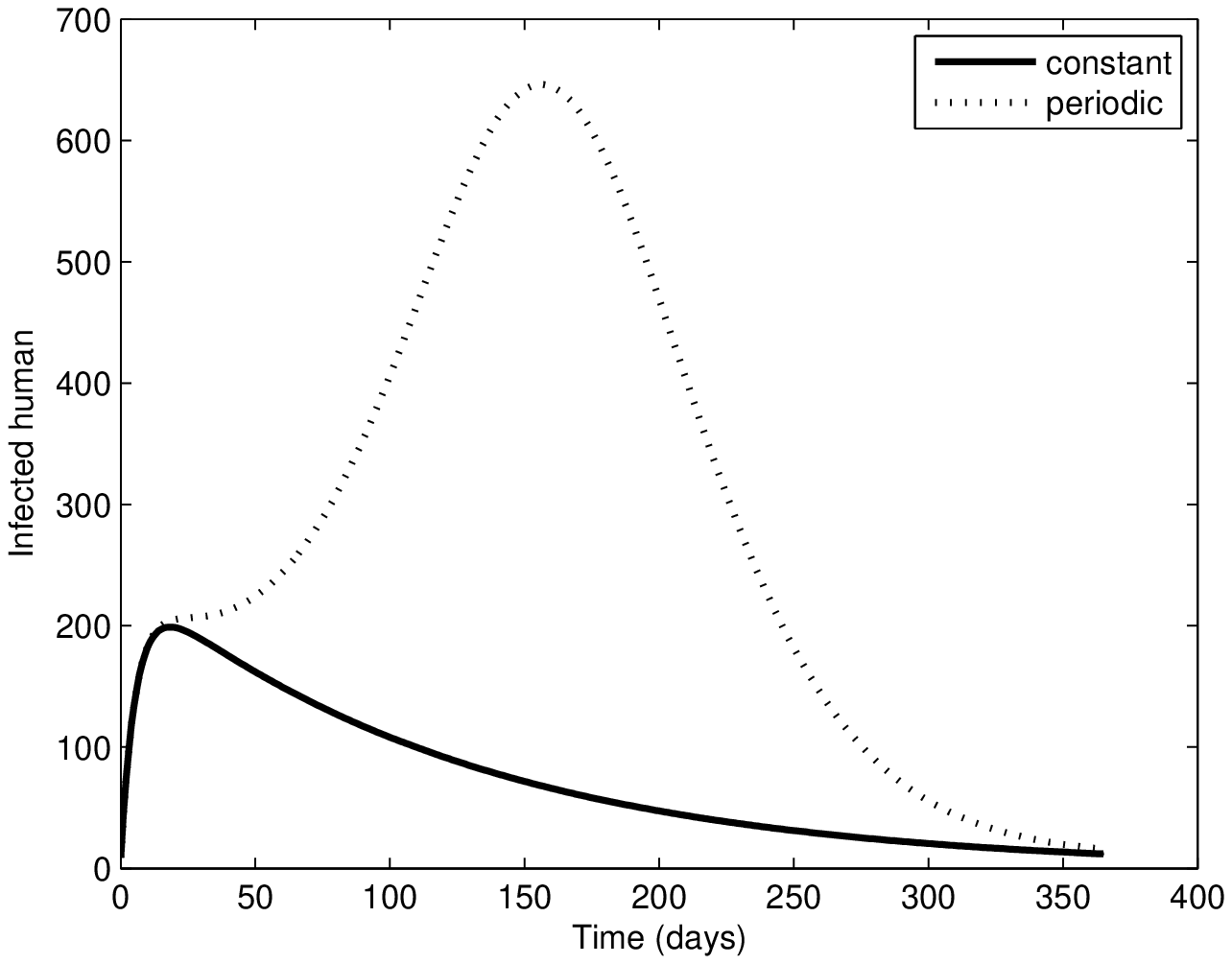}
\caption{Scenario 1}
\label{ih_mix_low}
\end{subfigure}
\begin{subfigure}[b]{0.4\textwidth}
\includegraphics[scale=0.45]{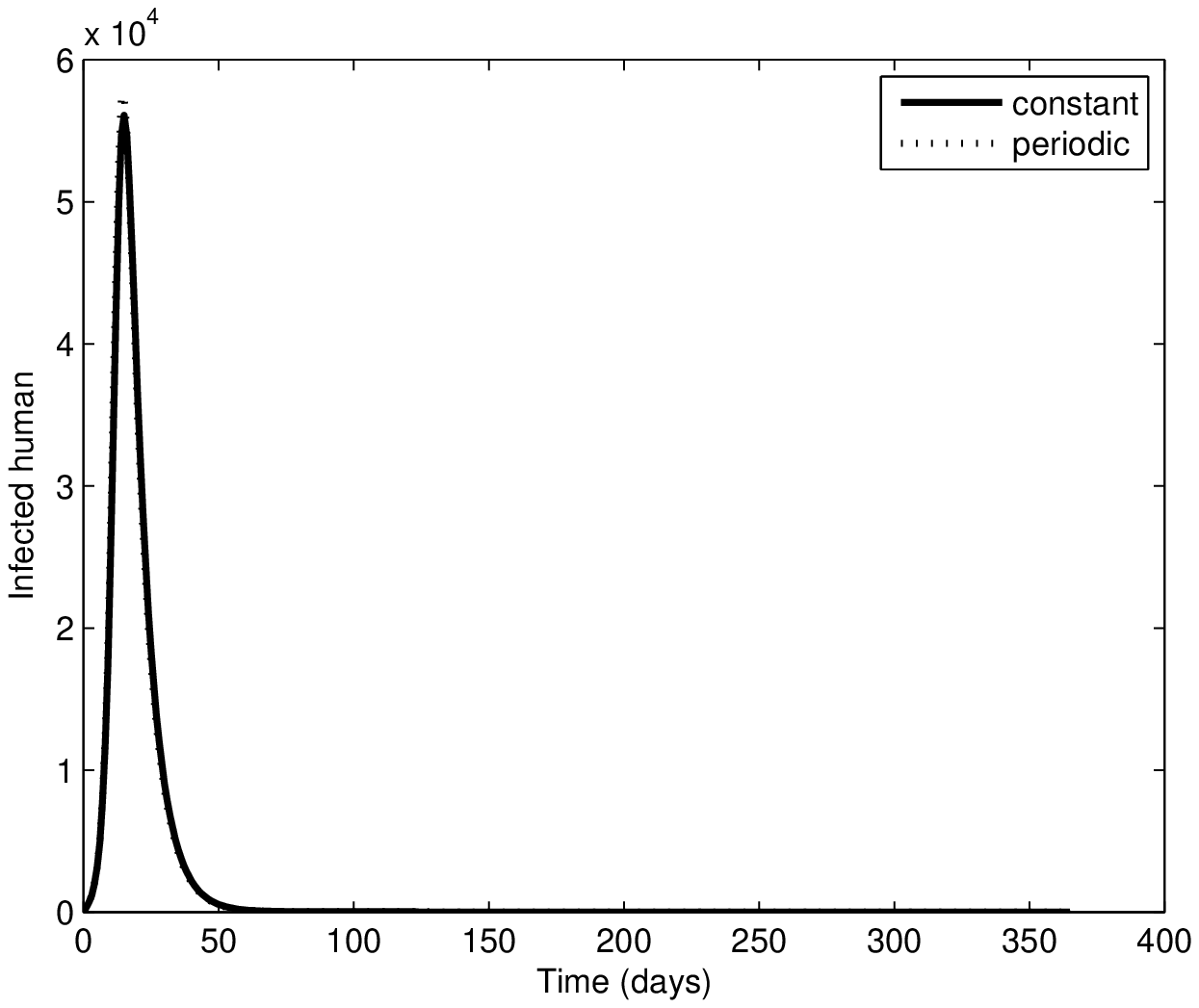}
\caption{Scenario 2}
\label{ih_mix_high}
\end{subfigure}
\begin{subfigure}[b]{0.4\textwidth}
\includegraphics[scale=0.45]{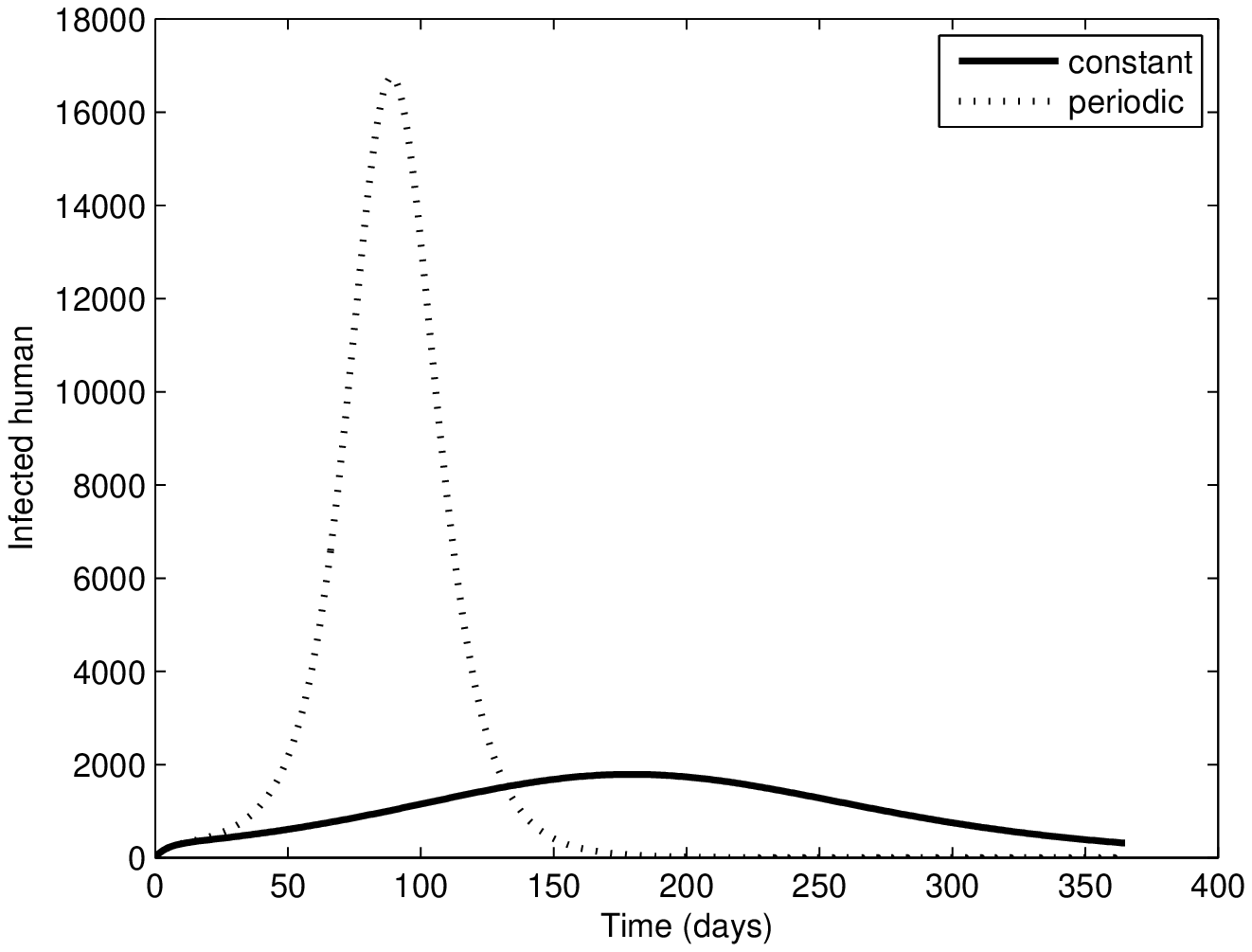}
\caption{Scenario 3}
\label{ih_mix_madeira}
\end{subfigure}
\caption{Infected human using a constant and a periodic mosquito population
(solid and dashed lines, respectively).}\label{ih_mix}
\end{figure}

With the aim of fighting the disease, reducing simultaneously
the costs with infected individuals and the costs of insecticide
campaigns to kill the mosquito, an optimal control problem is presented
and analysed in the next section.


\section{Optimal control problem}
\label{sec:4}

The control strategies for the reduction of infected individuals imply a cost
of implementation. This cost can be modeled through the formulation of an optimal
control problem, mathematically traduced by adding a functional.

Let us consider the previous differential system \eqref{cap6_ode1}--\eqref{cap6_ode2} with
constant mosquito population. To this, we add to the mosquito compartments the control insecticide, $u(t)$,
with $0\leq u(t)\leq 1$.  Thus, the second part of the differential system is rewritten as
\begin{equation}
\label{cap6_ode2_control}
\begin{cases}
\frac{dS_m(t)}{dt} = \mu_m N_m
-\left(B \beta_{hm}\frac{I_h(t)}{N_h}+\mu_m +u(t)\right) S_m(t),\\
\frac{dI_m(t)}{dt} = B \beta_{hm}\frac{I_h(t)}{N_h}S_m(t)
-\left(\mu_m+u(t)\right) I_m(t).
\end{cases}
\end{equation}
Our aim is to minimize the number of infected individuals, $I_h$, while keeping
the cost of control strategy implementation low, that is, we want to minimize
\begin{equation}
\label{functional}
\displaystyle C\left(u\right)=\int_{0}^{t_f}\left[\gamma_D I_h(t)
+\gamma_S u^2(t)\right]dt,
\end{equation}
where the coefficients $\gamma_D$ and $\gamma_S$ represent the balancing cost factors for
infected individuals and spraying campaigns, respectively. More precisely,
the optimal control problem consists in finding a control $u^{*}$ such that the associated
state trajectory $\left(S_h^*, I_h^*, R_h^*, S_m^*, I_m^*\right)$ is solution of the
control system \eqref{cap6_ode1} and \eqref{cap6_ode2_control} in the interval $[0, t_f]$
with the initial conditions \eqref{initial_conditions} and minimizing the cost functional $C$:
\begin{equation}
\label{functional_minimization}
C(u^*)=\min_{u \in \Omega} C(u),
\end{equation}
where $\Omega$ is the set of admissible controls given by
$\Omega = \{ u \in L^1(0,t_f) \, | \, 0\leq u(t) \leq 1\}$.
The existence of the optimal control $u^*(\cdot)\in \Omega$
comes from the convexity of the cost functional
\eqref{functional} with respect to the control and the
regularity of the system \eqref{cap6_ode1} and \eqref{cap6_ode2_control}
(see, e.g., \cite{Cesari1983} for existence results of optimal solutions).
According to the Pontryagin minimum principle \cite{Pontryagin1962},
if $u^*(\cdot)$  is optimal for the problem \eqref{functional_minimization}, \eqref{cap6_ode1} and \eqref{cap6_ode2_control}
with the initial conditions given by \eqref{initial_conditions} and fixed final time $t_f$,
then there exists an adjoint vector, $\lambda=\left(\lambda_1(t),\lambda_2(t),
\lambda_3(t),\lambda_4(t),\lambda_5(t) \right)$, such that
\begin{equation*}
\dot{S}_h=\frac{\partial H}{\partial \lambda_1},
\quad \dot{I}_h=\frac{\partial H}{\partial \lambda_2},
\quad \dot{R}_h=\frac{\partial H}{\partial \lambda_3},
\quad \dot{S}_m=\frac{\partial H}{\partial \lambda_4},
\quad \dot{I}_m=\frac{\partial H}{\partial \lambda_5}
\end{equation*}
and
\begin{equation*}
\dot{\lambda}_1=-\frac{\partial H}{\partial S_h},
\quad \dot{\lambda}_2=-\frac{\partial H}{\partial I_h},
\quad \dot{\lambda}_3=-\frac{\partial H}{\partial R_h},
\quad \dot{\lambda}_4=-\frac{\partial H}{\partial S_m},
\quad \dot{\lambda}_5=-\frac{\partial H}{\partial I_m},
\end{equation*}
where function $H$, called the Hamiltonian, is defined by
\begin{equation*}
\begin{split}
H = & \gamma_D I_h+\gamma_S u^2\\
& + \lambda_1 \left( \mu_h N_h - \left(B\beta_{mh}\frac{I_m}{N_h}+\mu_h\right)S_h\right)\\
& + \lambda_2 \left( B\beta_{mh}\frac{I_m}{N_h}S_h -(\eta_h+\mu_h) I_h\right)\\
& + \lambda_3 \left( \eta_h I_h - \mu_h R_h\right)\\
& + \lambda_4 \left( \mu_m N_m -\left(B \beta_{hm}\frac{I_h}{N_h}+\mu_m \right) S_m\right)\\
& + \lambda_5 \left(  B \beta_{hm}\frac{I_h}{N_h}S_m -\mu_m I_m\right).
\end{split}
\end{equation*}
Moreover, the minimality condition
\begin{multline*}
H\left(S_h^*(t), I_h^*(t), R_h^*(t), S_m^*(t), I_m^*(t), \lambda(t),u^*(t)\right)\\
= \min_{0\leq u \leq 1} H \left(S_h^*(t), I_h^*(t), R_h^*(t), S_m^*(t), I_m^*(t), \lambda(t),u\right)
\end{multline*}
holds almost everywhere on $[0, t_f]$ together with the transversality conditions
\begin{equation*}
\lambda_i(t_f)=0, \quad i=1,\ldots,5.
\end{equation*}

\begin{theorem}
The optimal control problem \eqref{functional_minimization}, \eqref{cap6_ode1},
\eqref{cap6_ode2_control} with fixed initial conditions
\eqref{initial_conditions} and fixed final time $t_f$
admits an unique solution $\left(S_h^*, I_h^*, R_h^*, S_m^*, I_m^*\right)$ associated
to an optimal control $u^{*}(\cdot)$ on $[0,t_f]$. Moreover, there exists adjoint functions
$\lambda_1^{*}(\cdot)$, $\lambda_2^{*}(\cdot)$, $\lambda_3^{*}(\cdot)$, $\lambda_4^{*}(\cdot)$ and
$\lambda_5^{*}(\cdot)$ such that
\begin{equation}
\label{adjoint_equations}
\begin{cases}
\dot{\lambda}_1^{*}(t)=(\lambda_1^{*}(t)-\lambda_2^{*}(t))B \beta_{mh}I_m^*(t) +\lambda_1^*(t)\mu_m\\
\dot{\lambda}_2^{*}(t)=(\lambda_4^{*}(t)-\lambda_5^{*}(t))B \beta_{hm}I_h^*(t) + \lambda_2^*(t)(\eta_h+\mu_h)\\
\dot{\lambda}_3^{*}(t)=\lambda_3^*(t)\mu_h\\
\dot{\lambda}_4^{*}(t)=(\lambda_4^{*}(t)-\lambda_5^{*}(t))B \beta_{hm}I_h^*(t)+\lambda_4^*(t)(\mu_m+u^*(t))\\
\dot{\lambda}_5^{*}(t)=(\lambda_1^{*}(t)-\lambda_2^{*}(t))B \beta_{mh}S_h^*(t)+\lambda_5^*(\mu_m+u^*(t))
\end{cases}
\end{equation}
with transversality conditions
\begin{equation}
\label{tranversality}
\lambda_i^{*}(t_f)=0, \quad i=1,\ldots 5.
\end{equation}
Furthermore,
\begin{equation}
\label{optimal_control}
u^{*}(t)= \min \left\{\max\left\{0,\frac{\lambda_4^*(t) S_m^*(t)+\lambda_5^*(t) I_m^*(t)}{2 \gamma_S}\right\},1\right\}.
\end{equation}
\end{theorem}

\begin{proof}
Similar to the one found in \cite{Silva2012}.
\end{proof}

For the optimal control problem, we used the parameter values of Scenario~2,
because it is the case more threatening for public health and, therefore, the one that all efforts
must be invested. For the first two graphics (Figures~\ref{infected_oc} and \ref{control_oc}),
both balancing costs, $\gamma_D$ and $\gamma_S$, assume the value one.
The problem was solved in \texttt{Matlab}, using the forward-backward sweep method \cite{Lenhart2007}.
The process begins with an initial guess on the control variable. Then, the state
equations are simultaneously solved forward in time and the adjoint
equations are solved backward in time.
The control is updated by inserting the new values of states and
adjoints into its characterization, and the
process is repeated until convergence occurs.

We start showing that the implementation of the control has a positive impact
on the reduction of infected individuals. Figure~\ref{infected_oc}
reports that the fraction of infected individuals significantly decreases
when control strategies are implemented. More precisely, with the
control strategy, the infected people is near to zero after 150 days, whereas
without control the outbreak lasts more than one year. To minimize the total
number of infectious, the optimal control $u^{*}$ is applied, which
decreases to the lower bound at the end of the year
(see Figure~\ref{control_oc}).
\begin{figure}[ptbh]
\centering
\includegraphics[scale=0.5]{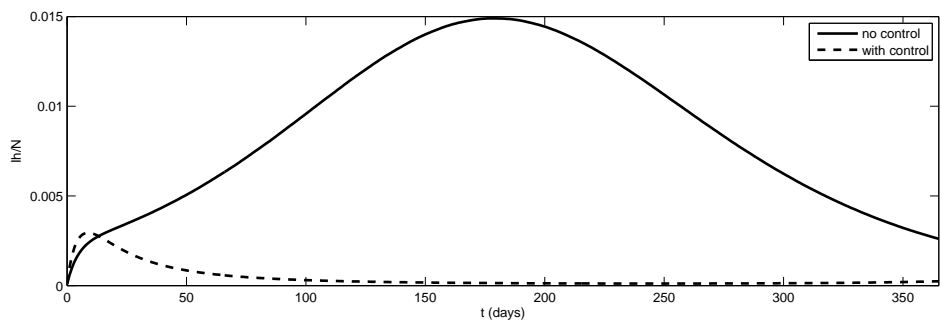}
{\caption{\label{infected_oc} Fraction of infected individuals
$I_h/N_h$ with and without control, where $\gamma_D=\gamma_S=1$.}}
\end{figure}
\begin{figure}[ptbh]
\centering
\includegraphics[scale=0.5]{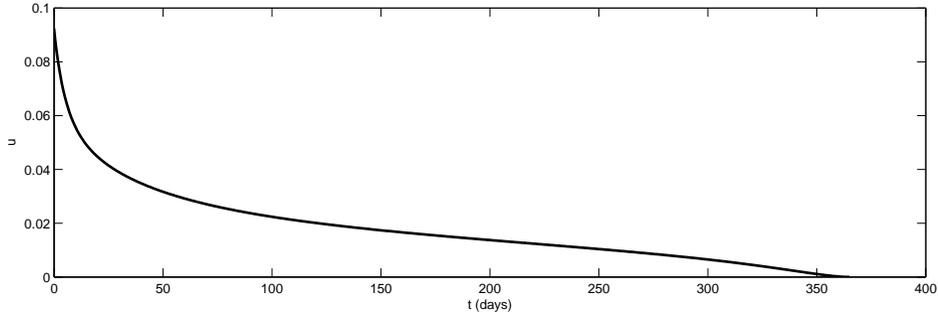}
{\caption{\label{control_oc} Optimal control $u^{*}$ with $\gamma_D=\gamma_S=1$.}}
\end{figure}

The relevance of the optimal control strategies were tested, in the
reduction of the fraction of infected individuals $I_h/N_h$
(Figures~\ref{infected_oc_different_gamma} and \ref{control_oc_different_gamma}).
Three bioeconomic approaches were simulated:
the first one where both human and economic factors are considered
($\gamma_D=\gamma_S=1$); the second one where the human factor is preponderant
($\gamma_D=1$ and $\gamma_S=0$); and the last one, where the main issue is the economic
costs with insecticide ($\gamma_D=0$ and $\gamma_S=1$). As expected, the higher values
for infected individuals are reached in the third bioeconomic approach. The same approach,
presents the optimal curve for $u$ close to zero, because it is expensive to apply insecticide.
\begin{figure}[ptbh]
\centering
\includegraphics[scale=0.5]{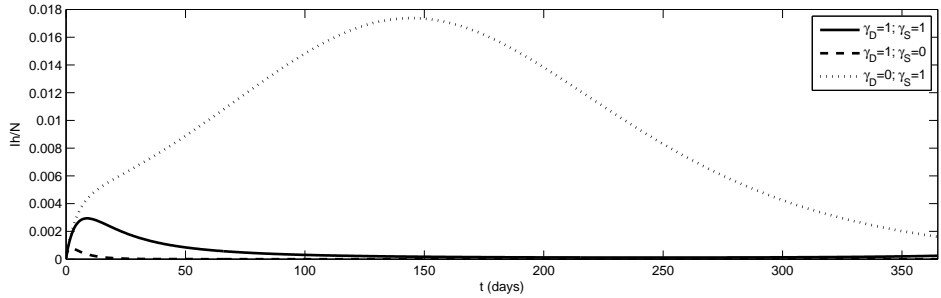}
{\caption{\label{infected_oc_different_gamma}
Fraction of infected individuals $I_h/N_h$ with distinct bioeconomic approaches.}}
\end{figure}
\begin{figure}[ptbh]
\centering
\includegraphics[scale=0.5]{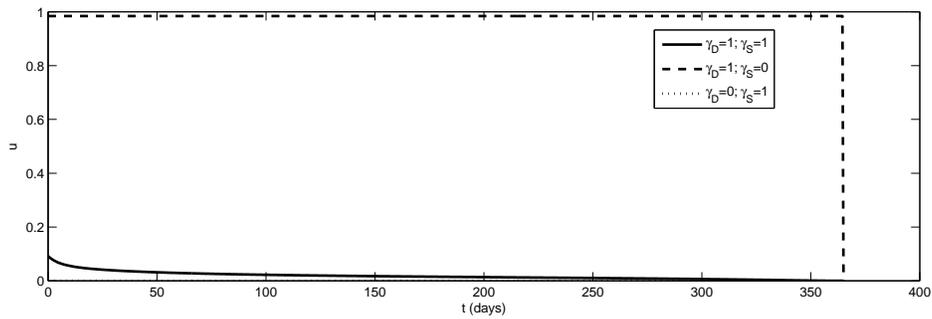}
{\caption{\label{control_oc_different_gamma} Optimal control $u^{*}$
with distinct bioeconomic approaches.}}
\end{figure}


\section{Conclusion}
\label{sec:5}

Temperature and rainfall can be either an effective barrier or a facilitator
of vector-borne diseases. The ambient temperature increased over the last few decades,
and may contribute to the drastic increase of dengue cases.
In this paper, we showed that small changes in the parameters of the model,
related to vectorial competence, can provoke great changes in the study of
dengue disease. In most regions where the disease is present, there is at least
two seasons, with distinct temperature and humidity. In this way,
a periodic function that allows to fit the mosquito population along the year
can be an interesting tool to the design of mathematical models.

Epidemiological modeling has largely focused on identifying the mechanisms responsible for epidemics
but has taken little account on economic constraints in analyzing control strategies. Economic models
have given insight into optimal control under constraints imposed by limited resources, but they frequently
ignore the spatial and temporal dynamics of the disease. Nowadays, the combination of epidemiological
and economic factors is essential. Therefore, we varied the cost functional,
giving a different answer depending on the main goal to reach,
thinking in economical or human-centered perspectives.


\section*{Acknowledgements}

This work was supported by Portuguese funds through
\emph{The Portuguese Foundation for Science and Technology} (FCT).
Rodrigues and Torres are supported by the
\emph{Center for Research and Development in Mathematics and Applications}
(CIDMA) within project PEst-OE/MAT/UI4106/2014;
Monteiro is supported by the ALGORITMI R\&D Center and project
PEST-OE/EEI/UI0319/2014.


\small



\end{document}